\documentclass[12pt,reqno]{amsart}

\usepackage{amsmath,amsfonts,amsthm,amssymb,amsxtra}
\usepackage{bbm} 
\usepackage{dsfont}
\usepackage[hidelinks]{hyperref} 
\usepackage{color}
\usepackage{comment}
\usepackage[numbers, sort]{natbib}
\usepackage{enumitem}
\usepackage{moreenum}
\usepackage[greek,english]{babel}

\numberwithin{equation}{section}
\newtheorem{theorem}{Theorem}[section]
\newtheorem{proposition}[theorem]{Proposition}

\theoremstyle{definition}

\theoremstyle{remark}

\newtheorem{remark}[theorem]{Remark}

\newcommand{\1}{\mathds{1}}

\renewcommand{\epsilon}{\varepsilon}

\newcommand{\N}{\mathbb{N}}

\renewcommand{\phi}{\varphi}
\newcommand{\R}{\mathbb{R}}

\newcommand{\Haus}{\mathcal{H}}

\DeclareMathOperator{\dist}{dist}
\DeclareMathOperator{\diam}{diam}

\DeclareMathOperator{\Tr}{Tr}
\DeclareMathOperator{\tr}{Tr}

\newcommand{\limplus}{{\mathchoice{\vcenter{\hbox{$\scriptstyle +$}}}
		{\vcenter{\hbox{$\scriptstyle +$}}}
		{\vcenter{\hbox{$\scriptscriptstyle +$}}}
		{\vcenter{\hbox{$\scriptscriptstyle +$}}}
}}
\newcommand{\limminus}{{\mathchoice{\vcenter{\hbox{$\scriptstyle -$}}}
		{\vcenter{\hbox{$\scriptstyle -$}}}
		{\vcenter{\hbox{$\scriptscriptstyle -$}}}
		{\vcenter{\hbox{$\scriptscriptstyle -$}}}
}}

\begin{document}

\title[Riesz means of Laplace eigenvalues]{An asymptotic shape optimization problem\\ for Riesz means of Laplace eigenvalues}

\author{Rupert L. Frank}
\address[Rupert L. Frank]{Mathe\-matisches Institut, Ludwig-Maximilians Universit\"at M\"unchen, The\-resienstr.~39, 80333 M\"unchen, Germany, and Munich Center for Quantum Science and Technology, Schel\-ling\-str.~4, 80799 M\"unchen, Germany, and Mathematics 253-37, Caltech, Pasa\-de\-na, CA 91125, USA}
\email{r.frank@lmu.de}

\author{Simon Larson}
\address{\textnormal{(Simon Larson)} Mathematical Sciences, Chalmers University of Technology and the University of Gothenburg, SE-41296 Gothenburg, Sweden}
\email{larsons@chalmers.se}

\begin{abstract}
	We review our recent results on the problem of optimizing Riesz means of Laplace eigenvalues among convex sets of given measure in the regime where the cut-off parameter in the definition of the Riesz means tends to infinity. We show that for a certain range of Riesz exponents, the optimizing sets converge to a ball. We also present some new results where we optimize over disjoint unions of convex sets.
\end{abstract}

\newcommand\blfootnote[1]{%
	\begingroup
	\renewcommand\thefootnote{}\footnote{#1}%
	\addtocounter{footnote}{-1}%
	\endgroup
}


\thanks{\copyright\, 2026 by the authors. This paper may be reproduced, in its entirety, for non-commercial purposes.\\
	Partial support through US National Science Foundation grant DMS-1954995 (R.L.F.), the German Research Foundation grants EXC-2111-390814868 and TRR 352-Project-ID 470903074 (R.L.F.), the Knut and Alice Wallenberg foundation grant KAW 2017.0295 (S.L.), as well as the Swedish Research Council grant no.~2023-03985 (S.L.) is acknowledged.}

\maketitle

\section{Introduction}

This note is based on a talk by the first author at the QMath conference in Munich in 2025 and he is grateful to the organizers for the invitation to speak and to contribute to these proceedings. The talk was based on results obtained by both authors in the papers \cite{FrankLarson_Crelle20,FrankLarson_24,FrankLarson_CPAM26,FrankLarson_heat}. In addition to summarizing those, we also include a number of results that appeared in slightly different form in the preprint version of \cite{FrankLarson_CPAM26}, but not in the final version.

\bigskip

Let $\Omega\subset\R^d$, $d\geq 2$, be an open set and let $-\Delta_\Omega^{\rm D}$ and $-\Delta_\Omega^{\rm N}$ denote the Dirichlet and Neumann Laplacians on $\Omega$, respectively. We write $-\Delta_\Omega^\sharp$ when we refer to either one of these two operators.

For $\gamma>0$ and $\lambda\geq 0$ we are interested in the Riesz means
$$
\tr(-\Delta_\Omega^\sharp - \lambda)_\limminus^\gamma \,.
$$
The Riesz mean $\Tr(-\Delta_{\Omega}^\sharp-\lambda)_\limminus^0$ is defined as the dimension of the range of the spectral projection $\1_{[0,\lambda)}(-\Delta_\Omega^\sharp)$. (This choice is somewhat arbitrary. Our main results would not change if instead we would use the projection $\1_{[0,\lambda]}(-\Delta_\Omega^\sharp)$.)

More specifically, we are interested in the shape optimization problems
\begin{align*}
	& \sup\bigl\{\Tr(-\Delta^{\rm D}_\Omega-\lambda)_\limminus^\gamma:\ \Omega \subset\R^d \, \text{open} \,,\ |\Omega|=1\bigr\}\,,\\
	& \inf\bigl\{\Tr(-\Delta^{\rm N}_\Omega-\lambda)_\limminus^\gamma:\ \Omega \subset\R^d \, \text{open} \,,\ |\Omega|=1\bigr\}\,.
\end{align*}	
Note that it does not represent a loss of generality to normalize the measure of $\Omega$ to be one. Any other value can be reduced to this case by rescaling and adjusting $\lambda$.

We are interested in the behavior of these optimization problems in the limit $\lambda\to\infty$ and, in particular, in the question:
\begin{quote}
	Do optimizers to the above optimization problems converge, up to translations, to a ball as $\lambda\to 0$?
\end{quote}
This question remains open in this generality and the topic of this note is to present some results when the optimization is restricted to certain subclasses of open sets.

Before describing these results in detail, however, we explain why one might think that the above question can be answered affirmatively. Let us assume temporarily that $\gamma>0$. Then the two-term Weyl asymptotics state that
\begin{equation*}
	\begin{aligned}
		\Tr(-\Delta^{\rm D}_\Omega-\lambda)_\limminus^\gamma = L_{\gamma,d}^{\rm sc} |\Omega| \lambda^{\gamma+\frac d2} - \frac14 L_{\gamma,d-1}^{\rm sc} \mathcal H^{d-1}(\partial\Omega) \lambda^{\gamma+\frac{d-1}{2}} + o(\lambda^{\gamma+\frac{d-1}{2}}) \,,\\
		\Tr(-\Delta^{\rm N}_\Omega-\lambda)_\limminus^\gamma = L_{\gamma,d}^{\rm sc} |\Omega| \lambda^{\gamma+\frac d2} + \frac14 L_{\gamma,d-1}^{\rm sc} \mathcal H^{d-1}(\partial\Omega) \lambda^{\gamma+\frac{d-1}{2}} + o(\lambda^{\gamma+\frac{d-1}{2}}) \,.
	\end{aligned}
\end{equation*}
Here $L_{\gamma,d}^{\rm sc}$ are certain constants given explicitly in \eqref{eq: semiclassical constant} below, $\mathcal H^{d-1}(\partial\Omega)$ denotes the surface area of the boundary of $\Omega$ and we emphasize the different signs of the subleading term depending on the boundary condition. Typically, these asymptotics are proved for a given set $\Omega$, taking the limit $\lambda\to\infty$, but it is clear from these proofs that they are valid when $\Omega$ depends in a sufficiently regular way on $\lambda$. We assume that this sufficiently regular dependence is satisfied for the family $\Omega_\lambda$ formed by optimizers for our shape optimization problem. We see that, because of the volume normalization, the leading term is $L_{\gamma,d}^{\rm sc} \lambda^{\gamma+\frac d2}$. In order to optimize the subleading term, one wants to maximize $- \frac14 L_{\gamma,d-1}^{\rm sc} \mathcal H^{d-1}(\partial\Omega) \lambda^{\gamma+\frac{d-1}{2}}$ in the Dirichlet case and to minimize $+\frac14 L_{\gamma,d-1}^{\rm sc} \mathcal H^{d-1}(\partial\Omega) \lambda^{\gamma+\frac{d-1}{2}}$ in the Neumann case. Thus, in both cases one wants to minimize $\mathcal H^{d-1}(\partial\Omega)$ among sets with $|\Omega|=1$. By the isoperimetric inequality, the unique solution to this optimization problem is a ball, which motivates why one might be tempted to expect an affirmative answer to the above question.

The subtlety with this argument is that very little is known about the dependence of the optimizers on $\lambda$. In particular the `sufficiently regular dependence on $\lambda$' is unknown.

\medskip

Following \cite{LarsonJST} we restrict ourselves to the optimization problem in much smaller subclasses of open sets. In \cite{FrankLarson_Crelle20,FrankLarson_CPAM26} we considered the case where we optimize within the class of \emph{convex} sets. These results are summarized in Section \ref{sec:convex}. Then, in Section~\ref{sec:multicomp}, we consider the class of sets that are \emph{disjoint unions of convex open sets}.

We proceed to a simplified statement of the corresponding results. For $d\geq 1$ let $\mathcal{C}_d$ denote the collection of all open, bounded, non-empty, and convex subsets of $\R^d$. Let also $\widetilde{\mathcal{C}}_d$ denote the collection of all open, non-empty subsets of $\R^d$ whose connected components belong to $\mathcal{C}_d$. The first pair of optimization problems is
\begin{equation}
	\label{eq:defopt}
	\begin{split}
	M^{\rm D}_\gamma(\lambda)&:=\sup\bigl\{\Tr(-\Delta^{\rm D}_\Omega-\lambda)_\limminus^\gamma: \Omega \in \mathcal{C}_d, |\Omega|=1\bigr\}\,,\\
	M^{\rm N}_\gamma(\lambda)&:=\inf\bigl\{\Tr(-\Delta^{\rm N}_\Omega-\lambda)_\limminus^\gamma: \Omega \in \mathcal{C}_d, |\Omega|=1\bigr\}\,,
	\end{split}
\end{equation}	
and the second pair is
\begin{equation}
	\label{eq:defopttilde}
	\begin{split}
	\widetilde M_\gamma^{\rm D}(\lambda)&:=\sup\bigl\{\Tr(-\Delta_\Omega^{\rm D}-\lambda)_\limminus^\gamma: \Omega \in \widetilde{\mathcal{C}}_d, |\Omega|=1\bigr\}\,,\\
	\widetilde M_\gamma^{\rm N}(\lambda)&:=\inf\bigl\{\Tr(-\Delta_\Omega^{\rm N}-\lambda)_\limminus^\gamma: \Omega \in \widetilde{\mathcal{C}}_d, |\Omega|=1\bigr\}\,.
	\end{split}
\end{equation}

The following are our main results.

\begin{theorem}\label{main}
	Let $\sharp\in\{{\rm D},{\rm N}\}$. Let $\gamma>0$ if $d=2$ and $\gamma\geq 1/2$ if $d\geq 3$. Let $\{\lambda_j\}_{j\geq 1}\subset (0, \infty)$ be a sequence with $\lim_{j\to \infty}\lambda_j = \infty$ and let $\{\Omega_j\}_{j\geq 1} \subset\mathcal C_d$ be a sequence with $|\Omega_j|=1$ satisfying
	$$
	\lim_{j\to\infty} \lambda_j^{-\gamma-\frac{d-1}{2}} \left( \Tr(-\Delta_{\Omega_j}^\sharp-\lambda_j)_\limminus^\gamma - M_{\gamma}^{\sharp}(\lambda_j) \right) = 0 \,.
	$$
	Then $\{\Omega_j\}_{j\geq 1}$ converges, up to translations, in Hausdorff sense to a ball of unit measure.
\end{theorem}

\begin{theorem}\label{main2}
	Let $\sharp\in\{{\rm D},{\rm N}\}$. Let $\gamma\geq 1$. Let $\{\lambda_j\}_{j\geq 1}\subset (0, \infty)$ be a sequence with $\lim_{j\to \infty}\lambda_j = \infty$ and let $\{\widetilde\Omega_j\}_{j\geq 1} \subset\widetilde{\mathcal C}_d$ be a sequence with $|\widetilde\Omega_j|=1$ satisfying
	$$
	\lim_{j\to\infty} \lambda_j^{-\gamma-\frac{d-1}{2}} \left( \Tr(-\Delta_{\widetilde\Omega_j}^\sharp-\lambda_j)_\limminus^\gamma - \widetilde M_{\gamma}^{\sharp}(\lambda_j) \right) = 0 \,.
	$$
	Then $\{\widetilde\Omega_j\}_{j\geq 1}$ converges, up to translations, in Hausdorff sense to a ball of unit measure.
\end{theorem}

Both theorems are applicable, in particular, when $\Omega_j$ and $\widetilde\Omega_j$ are optimizers for $M_{\gamma}^{\sharp}(\lambda_j)$ and $\widetilde M_{\gamma}^{\sharp}(\lambda_j)$. (The existence of an optimizer for $M_{\gamma}^{\sharp}(\lambda)$ is proved in \cite[Lemma 3.1]{LarsonJST} and \cite[Lemma 2.6]{FrankLarson_CPAM26}.) We emphasize, however, that the theorem does not require the exact minimality. A certain notion of almost minimality suffices for the conclusion to hold.

We also stress that Theorem \ref{main} covers the full range $\gamma>0$ in dimension $d=2$. In the remaining cases we can only prove a conditional result.

\begin{theorem}\label{main3}
	Let $\sharp\in\{{\rm D},{\rm N}\}$ and let $d\geq 3$. Assume that
	\begin{align*}
		& \Tr(-\Delta_{\omega}^{\rm D}-\lambda)_\limminus^{1/2} \leq L_{1/2,d-1}^{\rm sc} |\omega|\lambda^{\frac d2} & \text{if}\ \sharp={\rm D} \,, \\
		& \Tr(-\Delta_{\omega}^{\rm N}-\lambda)_\limminus^{1/2} \geq L_{1/2,d-1}^{\rm sc} |\omega|\lambda^{\frac d2} & \text{if}\ \sharp={\rm N} \,,
	\end{align*}
	for any $\lambda\geq 0$ and $\omega\in\mathcal C_{d-1}$. Then the conclusion of Theorem \ref{main} holds for all $\gamma>0$.
\end{theorem}

\begin{theorem}\label{main4}
	Let $\sharp\in\{{\rm D},{\rm N}\}$ and let $d\geq 2$. Assume that
	\begin{align*}
		& \Tr(-\Delta_{\Omega}^{\rm D}-\lambda)_\limminus^0 \leq L_{0,d}^{\rm sc} |\Omega|\lambda^{\frac d2} & \text{if}\ \sharp={\rm D} \,, \\
		& \Tr(-\Delta_{\Omega}^{\rm N}-\lambda)_\limminus^0 \geq L_{0,d}^{\rm sc} |\Omega|\lambda^{\frac d2} & \text{if}\ \sharp={\rm N} \,,
	\end{align*}
	for any $\lambda\geq 0$ and $\Omega\in\mathcal C_{d}$. Then the conclusion of Theorem \ref{main2} holds for all $\gamma>0$.
\end{theorem}

We emphasize that the assumption in Theorem \ref{main4} is rather strong. It posits the validity of P\'olya's conjecture in the class of convex sets. By considering sets $\Omega=\omega\times(-\ell/2,\ell/2)$ with $\ell\gg\lambda^{-1/2}$ it is not hard to see that the assumption in Theorem~\ref{main4} implies the assumption in Theorem \ref{main3}.

In Sections \ref{sec:convex} and \ref{sec:multicomp} we will see that the assumptions in Theorems \ref{main3} and \ref{main4} are also necessary for the conclusion to hold. Thus, in some sense giving an affirmative answer to the question raised at the beginning of the introduction, namely showing convergence of optimizers of the $\widetilde M_{\gamma}^{\sharp}(\lambda)$-problem to balls, is as hard as proving P\'olya's conjecture for convex sets.

\subsection*{Bibliographic remarks}

Spectral shape optimization problems have a long and venerable history, with Rayleigh's conjecture and its proof by Faber and Krahn being an early highlight. We refer to~\cite{Henrot_17} and references therein for a review. Asymptotic questions similar to those studied in our paper have seen a surge of interest in the last decade, including \cite{AntunesFreitas_13,BucurFreitas_13, ColboisElSoufi_14,vdBerg_15, vdBergBucurGittins_16, vdBergGittins_17, GittinsLarson_17, Freitas_17, LarsonAFM, Lagace20, BuosoFreitas_20,Freitas_etal_21,Filonov_etal_Polya,BaurLarson_26}. Most of these results concern the case of optimizing individual eigenvalues (or the corresponding eigenvalue counting function). The paper \cite{LarsonJST} seems to have been among the first to study Riesz means.


\section{The critical Riesz exponent}

The semiclassical constant for $\gamma\geq 0$ and $d\in\N$ is defined by
\begin{equation}\label{eq: semiclassical constant}
	L_{\gamma, d}^{\rm sc} := \frac{\Gamma(1+\gamma)}{(4\pi)^{\frac{d}2}\Gamma(1+\gamma + \frac{d}{2})}\,.
\end{equation}
This constant appears in the Weyl asymptotics
\begin{equation}
	\label{eq:weyl}
	\lim_{\lambda\to\infty} \lambda^{-\gamma-\frac d2} \tr(-\Delta_\Omega^\sharp - \lambda)_\limminus^\gamma = L_{\gamma,d}^{\rm sc} \, |\Omega| \,.
\end{equation}
These asymptotics are valid for any open set of finite measure when $\sharp ={\rm D}$ and for bounded open sets with the extension property when $\sharp={\rm N}$; see \cite[Corollaries 3.17 and 3.21]{FrankLaptevWeidl} and references therein.

Let us introduce the excess factors
\begin{align*}
	r_{\gamma,d}^{\rm D} &:= \sup\biggl\{\frac{\Tr(-\Delta_\Omega^{\rm D}-\lambda)_\limminus^\gamma}{L_{\gamma,d}^{\rm sc}|\Omega|\lambda^{\gamma+\frac{d}2}}: \Omega \in \mathcal{C}_d, \lambda > 0\biggr\}\,,\\
	r_{\gamma,d}^{\rm N} &:= \inf\biggl\{\frac{\Tr(-\Delta_\Omega^{\rm N}-\lambda)_\limminus^\gamma}{L_{\gamma,d}^{\rm sc}|\Omega|\lambda^{\gamma+\frac{d}2}}: \Omega \in \mathcal{C}_d, \lambda > 0\biggr\}\,.
\end{align*}
That is, $r_{\gamma, d}^{\sharp}$ are 
the optimal constants $r_{\gamma,d}^\sharp$ so that, for all $\Omega\in \mathcal{C}_d$ and all $\lambda\geq 0$,
\begin{equation*}
	\Tr(-\Delta_\Omega^{\rm D}-\lambda)_\limminus^\gamma  \leq r_{\gamma,d}^{\rm D}L_{\gamma, d}^{\rm sc}|\Omega|\lambda^{\gamma+\frac{d}2}\quad  \mbox{and} \quad \Tr(-\Delta_\Omega^{\rm N}-\lambda)_\limminus^\gamma  \geq r_{\gamma,d}^{\rm N}L_{\gamma, d}^{\rm sc}|\Omega|\lambda^{\gamma+\frac{d}2}\,.
\end{equation*}
Let us summarize some facts about these constants. All these facts are valid also for the sharp constants in the corresponding inequalities without the convexity assumption on the underlying domain.
\begin{enumerate}[label=\textup{(}\hspace{-0.3pt}\alph*\textup{)}]
	\item $r_{\gamma,d}^{\rm D} <\infty$ and $r_{\gamma,d}^{\rm N}>0$; this is well known, see \cite[Corollaries 3.30 and 3.39]{FrankLaptevWeidl}.
	\item $r_{\gamma, d}^{\rm D}\geq 1 \geq r_{\gamma,d}^{\rm N}$; this is a consequence of Weyl's law \eqref{eq:weyl}.
	\item\label{itm: r properties monotonicity} $\gamma\mapsto r_{\gamma,d}^\sharp$ is non-increasing (for $\sharp=$ D) and non-decreasing (for $\sharp=$ N); this is a consequence of the Aizenman--Lieb argument; see e.g. \cite[Corollary 3.29]{FrankLaptevWeidl}.
	\item $\gamma\mapsto r_{\gamma,d}^\sharp$ is continuous; this follows from \cite[Lemmas 3.3 and 3.4]{FrankLarson_CPAM26}.
	\item $r_{\gamma, d}^{\rm D} = 1 = r_{\gamma,d}^{\rm N}$ for $\gamma\geq 1$; this is a consequence of the Berezin--Li--Yau and Kr\"oger inequalities \cite[Theorems 3.25 and 3.37]{FrankLaptevWeidl} and the monotonicity in \ref{itm: r properties monotonicity}.
\end{enumerate}

Define the critical exponents
\begin{equation*}
	\gamma_d^{\rm D} :=\inf\Bigl\{\gamma\geq 0: \Tr(-\Delta_\Omega^{\rm D}-\lambda)_\limminus^\gamma \leq L^{\rm sc}_{\gamma,d}|\Omega|\lambda^{\gamma+\frac{d}2} \mbox{ for all }\Omega \in \mathcal{C}_d, \lambda \geq 0\Bigr\} \,\,
\end{equation*}
and
\begin{equation*}
	\gamma_d^{\rm N} :=\inf\Bigl\{\gamma\geq 0: \Tr(-\Delta_\Omega^{\rm N}-\lambda)_\limminus^\gamma \geq L^{\rm sc}_{\gamma,d}|\Omega|\lambda^{\gamma+\frac{d}2} \mbox{ for all }\Omega \in \mathcal{C}_d, \lambda \geq 0\Bigr\} \,.
\end{equation*}
The exponent $\gamma_d^\sharp$ can be characterized as the smallest number so that if $\gamma \geq \gamma_d^{\rm D}$, then
\begin{equation*}
	\Tr(-\Delta_\Omega^{\rm D}-\lambda)_\limminus^\gamma \leq L_{\gamma,d}^{\rm sc} |\Omega| \lambda^{\gamma+\frac{d}2}\,.
\end{equation*}
for all bounded, open, and convex $\Omega \subset \R^d$ and $\lambda \geq 0$.
If $\gamma \geq \gamma_d^{\rm N}$ then, similarly,
\begin{equation*}
	\Tr(-\Delta_\Omega^{\rm N}-\lambda)_\limminus^\gamma \geq L_{\gamma,d}^{\rm sc} |\Omega| \lambda^{\gamma+\frac{d}2}
\end{equation*}
for all bounded, open, and convex $\Omega \subset \R^d$ and $\lambda \geq 0$. In other words, the critical exponent $\gamma_d^\sharp$ can be characterized as the unique smallest number so that
$$
r_{\gamma,d}^\sharp =1
\qquad\text{for all}\
\gamma\geq \gamma^\sharp_d \,.
$$

The validity of P\'olya's conjecture for convex sets is equivalent to $\gamma_d^\sharp =0$. In particular, since P\'olya's conjecture holds in dimension $d=1$, we have
$$
\gamma_1^\sharp = 0 \,.
$$
The following is one of the main results in \cite{FrankLarson_CPAM26}.

\begin{theorem}\label{smallerthanone}
	For any $d$ and $\sharp\in\{{\rm D},{\rm N}\}$, we have
	$$
	\gamma_d^\sharp < 1\,.
	$$
\end{theorem}

Let us briefly describe some of the ideas that go into the proof of this theorem. We need to prove that there is a $\gamma<1$ such that, for all $\Omega\in\mathcal C_d$ and $\lambda\geq 0$,
$$
\Tr(-\Delta_\Omega^{\rm D}-\lambda)_\limminus^\gamma \leq L_{\gamma,d}^{\rm sc} |\Omega| \lambda^{\gamma+\frac{d}2} \leq \Tr(-\Delta_\Omega^{\rm N}-\lambda)_\limminus^\gamma\,.
$$
The \emph{uniform semiclassics} from \cite{FrankLarson_24} (see also Theorem \ref{thm: Weyl asymptotics convex} below) imply that for each $\gamma_0>0$ there is a constant $B_0$ (depending only on $d$ and $\gamma_0$) such that the desired inequality holds for every $\gamma\geq\gamma_0$ and $\lambda \geq B_0\, r_{\rm in}(\Omega)^{-2}$. Here, $r_{\rm in}(\Omega)$ denotes the inradius of a set $\Omega \subset \R^d$, that is, the radius of the largest ball contained in $\Omega$,
\begin{equation*}
	r_{\rm in}(\Omega) := \sup_{x\in \Omega}\dist(x, \Omega^c)\,.
\end{equation*}

Meanwhile, \emph{improved Berezin--Li--Yau/Kr\"oger inequalities} (see also \cite{FrankLarsonPfeiffer}) show that for every $B>0$ there are $c_B<1<C_B$ such that for all $\Omega\in\mathcal C_d$ and $\lambda \leq B\, r_{\rm in}(\Omega)^{-2}$
$$
\Tr(-\Delta_\Omega^{\rm D}-\lambda)_\limminus \leq c_B L_{1,d}^{\rm sc} |\Omega| \lambda^{1+\frac{d}2} < C_B L_{1,d}^{\rm sc} |\Omega| \lambda^{\gamma+\frac{d}2} \leq \Tr(-\Delta_\Omega^{\rm N}-\lambda)_\limminus\,.
$$
From the latter inequality one can deduce that for every $B>0$ there is a $\gamma_B<1$ such that for all $\gamma\geq\gamma_B$, for all $\Omega\in\mathcal C_d$ and $\lambda \geq B\, r_{\rm in}(\Omega)^{-2}$ one has the desired inequality. Combining these two facts, we obtain Theorem \ref{smallerthanone}.


\section{The optimization problem for convex sets}\label{sec:convex}

In this section we turn our attention to the shape optimization problems $M^{\rm D}_\gamma(\lambda)$ and $M^{\rm N}_\gamma(\lambda)$, introduced in \eqref{eq:defopt}. More precisely, for fixed $\gamma\geq 0$, we shall be interested in the numbers $M^\sharp_\gamma(\lambda)$ and the corresponding optimizers in the limit $\lambda\to\infty$.

Our first result concerns the asymptotics of $M_\gamma^\sharp(\lambda)$ as $\lambda\to\infty$. We find it remarkable that it is the critical Riesz exponent in dimension $d-1$ that is relevant for the shape optimization problem in dimension $d$.

\begin{proposition}\label{prop:shapeoptasymp r}
	Let $d\geq 2$, $\gamma\geq 0$ and $\sharp \in \{{\rm D}, {\rm N}\}$. Then
	\begin{equation*}
		\lim_{\lambda \to \infty}\frac{M_{\gamma}^{\sharp}(\lambda)}{L_{\gamma,d}^{\rm sc}\lambda^{\gamma+\frac{d}2}} = r_{\gamma+\frac12,d-1}^\sharp \,.
	\end{equation*}
\end{proposition}

In particular, we have, if $\gamma\geq \gamma_{d-1}^\sharp-\frac12$, then
\begin{equation*}
	\lim_{\lambda \to \infty}\frac{M_{\gamma}^{\sharp}(\lambda)}{L_{\gamma,d}^{\rm sc}\lambda^{\gamma+\frac{d}2}} =1 \,,
\end{equation*}
while if $\gamma < \gamma_{d-1}^\sharp -\frac12$, then
\begin{equation*}
	\begin{aligned}
		\lim_{\lambda \to \infty}\frac{M_{\gamma}^{\sharp}(\lambda)}{L_{\gamma,d}^{\rm sc}\lambda^{\gamma+\frac{d}2}}>1 \quad  \mbox{if }\sharp ={\rm D} \quad \mbox{and}\quad
		\lim_{\lambda \to \infty}\frac{M_{\gamma}^{\sharp}(\lambda)}{L_{\gamma,d}^{\rm sc}\lambda^{\gamma+\frac{d}2}}<1\quad  \mbox{if }\sharp ={\rm N}\,.
	\end{aligned}
\end{equation*}

Now we turn to the finer question of describing the asymptotic behavior of sets realizing the extremum in $M^\sharp_\gamma(\lambda)$ as $\lambda\to\infty$ or, more generally, of sets almost realizing the extremum. For the sake of simplicity we restrict ourselves to the case $\gamma>0$ and refer to \cite{FrankLarson_CPAM26} for the case $\gamma=0$.

To formulate our results we need to introduce some terminology. For bounded $\Omega, \Omega' \in \mathcal{C}_d$ the (complementary) Hausdorff distance between these sets is defined by
\begin{equation*}
	d^{H}(\Omega, \Omega') := \max\Bigl\{\sup_{x \in K\setminus\Omega}\dist(x, K\setminus\Omega'), \sup_{x \in K\setminus\Omega'}\dist(x, K\setminus\Omega)\Bigr\}\,,
\end{equation*}
where $K\subset \R^d$ is a compact set with $\Omega, \Omega' \subset K$. The definition is independent of the choice of $K$. For basic properties concerning the Hausdorff distance we refer to \cite{HenrotPierre_18}.

The following theorem describes the asymptotic behavior of (almost) optimizers to the problems $M_\gamma^\sharp(\lambda)$ as $\lambda\to\infty$.

\begin{theorem}\label{thm: shape optimization convex}
	Let $d\geq 2$, $\gamma> 0$ and $\sharp \in \{{\rm D}, {\rm N}\}$. Let $\{\lambda_j\}_{j\geq 1}\subset (0, \infty)$ be a sequence with $\lim_{j\to \infty}\lambda_j = \infty$ and let $\{\Omega_j\}_{j\geq 1} \subset\mathcal C_d$ be a sequence with $|\Omega_j|=1$ satisfying
	\begin{equation*}
		\lim_{j \to \infty} \frac{\Tr(-\Delta^{\sharp}_{\Omega_j}-\lambda_j)_\limminus^\gamma-M_\gamma^{\sharp}(\lambda_j)}{\lambda_j^{\gamma + \frac{d-1}{2}}} = 0\,.
	\end{equation*}
	\begin{enumerate}[label=\textup{(}\hspace{-0.3pt}\alph*\textup{)}]
		\item\label{itm: Shape opt thm super critical} If $\gamma >(\gamma_{d-1}^{\sharp}-\tfrac{1}2)_\limplus$, then the sequence $\{\Omega_j\}_{j\geq 1}$ converges, up to translations, with respect to the Hausdorff distance to a ball $B\subset\R^d$ with $|B|=1$ and
		\begin{equation*}
			M_\gamma^{\sharp}(\lambda) = \Tr(-\Delta_B^{\sharp}-\lambda)_\limminus^\gamma + o(\lambda^{\gamma+ \frac{d-1}2})
			\qquad\text{as}\ \lambda\to\infty \,.
		\end{equation*}
		\item\label{itm: Shape opt thm sub critical case} If $\gamma \leq  \gamma_{d-1}^{\sharp}-\frac12$, then
		\begin{equation*}
			\limsup_{j\to \infty}r_{\rm in}(\Omega_j)\sqrt{\lambda_j}<\infty
		\end{equation*}
		and there is a constant $C>0$ such that after appropriate translations and rotations
		\begin{equation*}
			\Omega_j \subset \{x = (x', x_d) \in \R^d: |x'| < C \, r_{\rm in}(\Omega_j)\} \quad \mbox{for all }j \geq 1\,.
		\end{equation*}
	\end{enumerate}
\end{theorem}

Note that Theorems \ref{main} and \ref{main3} in the introduction are consequences of part \ref{itm: Shape opt thm super critical} of Theorem \ref{thm: shape opt multicomponent}. Indeed, since $\gamma_d^\sharp<1$ by Theorem \ref{smallerthanone}, the assumption $\gamma>(\gamma_{d-1}^{\sharp}-\tfrac{1}2)_\limplus$ is trivially satisfied for $\gamma\geq 1/2$. (In fact, Theorem \ref{smallerthanone} is only needed for $\gamma=\frac12$. For $\gamma>\frac12$ the Berezin--Li--Yau/Kr\"oger inequalities, which imply $\gamma_d^\sharp\leq 1$, suffice.) Moreover, the assumption in Theorem \ref{main3} implies $\gamma_{d-1}^\sharp \leq 1/2$, so under this assumption the assumption $\gamma>(\gamma_{d-1}^{\sharp}-\tfrac{1}2)_\limplus$ is satisfied for any $\gamma>0$.


\subsection*{Sketches of proofs}

Proposition \ref{prop:shapeoptasymp r} and Theorem \ref{thm: shape optimization convex} are \cite[Proposition 6.1 and Theorem 1.7]{FrankLarson_CPAM26}, where full proofs are given. Here we only briefly sketch the proof of part \ref{itm: Shape opt thm super critical} of Theorem \ref{thm: shape optimization convex}.

Thus, let $\lambda_j$ and $\Omega_j$ be as in the theorem. The first step in the proof is to show that
\begin{equation}
	\label{eq:inradiusasymptotics}
	\lim_{j \to \infty} \sqrt{\lambda_j}\, r_{\rm in}(\Omega_j) = + \infty \,.
\end{equation}
To do so, we argue by contradiction. By considering the John ellipsoid associated to $\Omega_j$, we obtain, for each $j$, length scales $\ell_j^{(1)}\leq \ldots \leq \ell_j^{(d)}$ such that $\ell_j^{(1)} \sim r_{\rm in}(\Omega_j)$ and $\prod_{k=1}^d \ell_j^{(d)} \sim |\Omega_j|=1$. Therefore, after passing to a subsequence we may assume that there is an integer $1\leq m<d$ such that $\sqrt{\lambda_j}\, \ell^{(m)}_j$ has a positive, finite limit and $\sqrt{\lambda_j}\, \ell^{(m+1)}_j \to+\infty$. We now rescale $\Omega_j$ anisotropically in order to make its John ellipsoid into a unit ball and then we apply Blaschke's selection theorem to obtain a subsequence that converges to a limiting set. Thus, we are in the situation of a partially semi-classical limit, where $m$ directions degenerate. Spectral asymptotics in this limit are derived in \cite[Theorem 1.8]{FrankLarson_CPAM26}, where it is shown that if
\begin{equation*}
	\lim_{j \to \infty} \frac{\Tr(-\Delta_{\Omega_j}^\sharp-\lambda_j)_\limminus^\gamma}{L_{\gamma,d}^{\rm sc}|\Omega_j|\lambda_j^{\gamma+\frac{d}2}} 
\end{equation*}
exists, then there exists an open bounded convex non-empty set $\Omega_* \subset \R^d$ such that
\begin{equation*}
	\lim_{j \to \infty} \frac{\Tr(-\Delta_{\Omega_j}^\sharp-\lambda_j)_\limminus^\gamma}{L_{\gamma,d}^{\rm sc}|\Omega_j|\lambda_j^{\gamma+\frac{d}2}} = \frac{1}{L^{\rm sc}_{\gamma+ \frac{d-m}{2}, m}|\Omega_*|}\int_{P^\perp\Omega_*} \Tr(-\Delta_{\Omega_*(y)}^\sharp-1)_\limminus^{\gamma + \frac{d-m}{2}}\,dy \,,
\end{equation*}
where
\begin{align*}
	P^\perp\Omega_* &:= \{y\in \R^{d-m}: (x,y) \in \Omega_* \mbox{ for some } x \in \R^m\}\,,\\
	\Omega_*(y) &:= \{x \in \R^{m}: ( x,y) \in \Omega_*\}\,.
\end{align*}
More specifically, $\Omega_*$ is an anisotropic dilation of the bounded open convex non-empty set obtain by Blaschke's selection theorem. For more on partially semiclassical limits we refer to \cite{CarlenFrankLarson}.

We claim that
\begin{equation}
	\label{eq:strictineq}
	\gamma + \frac{d-m}{2} > \gamma_{m}^\sharp \,.
\end{equation}
Indeed, when $m\geq d-2$ this follows from $\gamma>0$ and $\gamma_m^\sharp\leq 1$ and when $m=d-1$ it follows from the assumption of part \ref{itm: Shape opt thm super critical} of Theorem \ref{thm: shape optimization convex}.

Using the strict inequality \eqref{eq:strictineq} one can show the strict inequality
$$
\Tr(-\Delta_{\Omega_*(y)}^{\rm D}-1)_\limminus^{\gamma + \frac{d-m}{2}} < L_{\gamma+\frac{d-m}2,m}^{\rm sc} |\Omega_*(y)|
$$
and the converse inequality for $\sharp = {\rm N}$. Thus, by Fubini's theorem,
$$
\frac{1}{L^{\rm sc}_{\gamma+ \frac{d-m}{2}, m}|\Omega_*|}\int_{P^\perp\Omega_*} \Tr(-\Delta_{\Omega_*(y)}^{\rm D}-1)_\limminus^{\gamma + \frac{d-m}{2}}\,dy
< \frac{1}{|\Omega_*|} \int_{P^\perp\Omega_*} |\Omega_*(y)|\,dy = 1
$$
and the converse inequality for $\sharp = {\rm N}$. This contradicts Proposition \ref{prop:shapeoptasymp r}, which says that
$$
\lim_{\lambda \to \infty}\frac{M_{\gamma}^{\sharp}(\lambda)}{L_{\gamma,d}^{\rm sc}\lambda^{\gamma+\frac{d}2}} =1 \,,
$$
Thus, we have proved \eqref{eq:inradiusasymptotics}.

The second step of the proof is to deduce from \eqref{eq:inradiusasymptotics} the convergence to a ball and the corresponding asymptotics of $M_\gamma^\sharp(\lambda)$. This follows from the uniform semiclassics for convex sets in \cite{FrankLarson_24}, which are summarized in the following theorem.

\begin{theorem}\label{thm: Weyl asymptotics convex}
	Let $d\geq 2$ and let $\Omega \subset \R^d$ be an open, bounded, and convex set. Then, for all $\lambda > 0$
	\begin{align*}
		\biggl|\Tr(-\Delta_\Omega^{\rm D}-&\lambda)_\limminus^\gamma - L_{\gamma, d}^{\rm sc} |\Omega| \lambda^{\gamma+\frac d2} + \frac{1}{4}L_{\gamma, d-1}^{\rm sc}\mathcal H^{d-1}(\partial\Omega)\lambda^{\gamma+\frac{d-1}2}\biggr| \hspace{135pt}\\
		&\leq C \mathcal H^{d-1}(\partial\Omega)\lambda^{\gamma+ \frac{d-1}{2}}\bigl(r_{\textup{in}}(\Omega)\sqrt{\lambda}\bigr)^{-\frac{\alpha}{11}}\,, 
    \end{align*}
    and
    \begin{align*}
		\biggl|\Tr(-\Delta_\Omega^{\rm N}-&\lambda)_\limminus^\gamma - L_{\gamma, d}^{\rm sc} |\Omega| \lambda^{\gamma+\frac d2} - \frac{1}{4}L_{\gamma, d-1}^{\rm sc}\mathcal H^{d-1}(\partial\Omega)\lambda^{\gamma+\frac{d-1}2}\biggr|\hspace{135pt} \\
        &\leq C \mathcal H^{d-1}(\partial\Omega)\lambda^{\gamma+ \frac{d-1}{2}}\Bigl[\bigl(1+\ln_\limplus\bigl(r_{\rm in}(\Omega)\sqrt{\lambda}\bigr)\bigr)^{-\alpha\max\{1, \gamma\}}+\bigl(r_{\textup{in}}(\Omega)\sqrt{\lambda}\bigr)^{1-d}\Bigr],
	\end{align*} 
	with
	\begin{equation*}
		\alpha = 1 \mbox{ for }\gamma \geq 1 \quad \mbox{and any} \quad \alpha \in (0, \gamma) \mbox{ for }0< \gamma< 1\,,
	\end{equation*}
	and where $C$ depends only on $\gamma, \alpha$, and the dimension.
\end{theorem}

The proof of this theorem is rather involved and we refer to \cite{FrankLarson_24} for the details. The Dirichlet case for $\gamma\geq 1$ had previously been treated in \cite{FrankLarson_Crelle20}. An important ingredient in the Neumann case are the heat kernel bounds from \cite{FrankLarson_heat}.

In the situation at hand, Theorem \ref{thm: Weyl asymptotics convex} implies
$$
\Tr(-\Delta_{\Omega_j}^\sharp-\lambda_j)_\limminus^\gamma = L_{\gamma,d}^{\rm sc} \lambda_j^{\gamma+\frac d2} \mp \tfrac14 L_{\gamma,d-1}^{\rm sc} \lambda_j^{\gamma+\frac{d-1}2} \mathcal H^{d-1}(\partial\Omega_j) + o(\lambda_j^{\gamma+\frac{d-1}2}) \,.
$$
with the upper and the lower sign in the Dirichlet and the Neumann case, respectively. Similarly, if $B$ is a ball with $|B|=1$, then
$$
\Tr(-\Delta_{B}^\sharp-\lambda_j)_\limminus^\gamma = L_{\gamma,d}^{\rm sc} \lambda_j^{\gamma+\frac d2} \mp \tfrac14 L_{\gamma,d-1}^{\rm sc} \lambda_j^{\gamma+\frac{d-1}2} \mathcal H^{d-1}(\partial B) + o(\lambda_j^{\gamma+\frac{d-1}2}) \,.
$$
Meanwhile, using $B$ as a competitor in the optimization problem and recalling the almost optimiality assumption on $\Omega_j$, we see that
\begin{align*}
	& \Tr(-\Delta_{B}^{\rm D} -\lambda_j)_\limminus^\gamma \leq M_\gamma^{\rm D}(\lambda_j) = \Tr(-\Delta_{\Omega_j}^{\rm D}-\lambda_j)_\limminus^\gamma + o(\lambda_j^{\gamma+\frac{d-1}2}) \,, \\
	& \Tr(-\Delta_{B}^{\rm N} -\lambda_j)_\limminus^\gamma \geq M_\gamma^{\rm D}(\lambda_j) = \Tr(-\Delta_{\Omega_j}^{\rm N}-\lambda_j)_\limminus^\gamma + o(\lambda_j^{\gamma+\frac{d-1}2}) \,.
\end{align*}
It follows that, in both cases,
$$
\mathcal H^{d-1}(\partial\Omega_j) \leq \mathcal H^{d-1}(\partial B) + o(1) \,.
$$

By the isoperimetric inequality, this implies that $\{\Omega_j\}_{j\geq 1}$ converges, up to a translation, in the Hausdorff sense to a ball. Indeed, if $\Omega$ denotes a limit point of $\{\Omega_j\}_{j\geq 1}$ (which exists by Blaschke's theorem), then $|\Omega|=1$ and $\mathcal H^{d-1}(\partial\Omega) \leq \mathcal H^{d-1}(\partial B)$. By the isoperimetric inequality, this forces $\Omega$ to be a ball, as claimed. 

Since $\mathcal H^{d-1}(\partial\Omega_j) = \mathcal H^{d-1}(\partial B) + o(1)$, the above asymptotics also give the desired asymptotics of $M_\gamma^\sharp$. This completes our sketch of proof of part \ref{itm: Shape opt thm super critical} in Theorem \ref{thm: shape optimization convex}. \qed


\section{A multi-component optimization problem}\label{sec:multicomp}

This section contains new material and we will provide complete proofs. We consider the pair of shape optimization problems $\widetilde M_\gamma^{\rm D}(\lambda)$ and $\widetilde M_\gamma^{\rm N}(\lambda)$, defined in \eqref{eq:defopttilde}.

While the results show some superficial similarity with those in the previous section, it is important to note that it is the critical exponent $\gamma_d^\sharp$ in dimension $d$, rather than $d-1$, that enters the results.

\begin{proposition}\label{prop: energy shape opt multicomponent}
	Fix $\sharp \in \{{\rm D}, {\rm N}\}$ and $\gamma \geq 0$. Then
	\begin{equation*}
		\lim_{\lambda\to \infty} \frac{\widetilde{M}^{\sharp}_\gamma(\lambda)}{L_{\gamma,d}^{\rm sc}\lambda^{\gamma+ \frac{d}2}} = r_{\gamma,d}^\sharp\,.
	\end{equation*}
\end{proposition}

\begin{theorem}\label{thm: shape opt multicomponent}
	Let $d\geq 2$, $\gamma\geq 0$ and $\sharp \in \{{\rm D}, {\rm N}\}$. Let $\{\lambda_j\}_{j\geq 1}\subset (0, \infty)$ be a sequence with $\lim_{j\to \infty}\lambda_j = \infty$ and let $\{\widetilde\Omega_j\}_{j\geq 1}$ be a sequence of sets in $\widetilde{\mathcal C}_d$ with $|\widetilde\Omega_j|=1$ satisfying
	\begin{equation*}
		\lim_{j \to \infty} \frac{\Tr(-\Delta^{\sharp}_{\widetilde\Omega_j}-\lambda_j)_\limminus^\gamma-\widetilde M^{\sharp}_\gamma(\lambda_j)}{\lambda_j^{\gamma + \frac{d-1}{2}}} = 0\,.
	\end{equation*}
	
	\begin{enumerate}[label=\textup{(}\hspace{-0.4pt}\alph*\textup{)}]
		\item\label{eq: multicomp supercritical} If $\gamma >\gamma_d^{\sharp}$, then up to translations the sequence $\{\widetilde\Omega_j\}_{j\geq 1}$ converges to a a single ball $B \subset \R^d$ with $|B|=1$ and
		\begin{equation*}
			\widetilde M^{\sharp}_\gamma(\lambda) = \Tr(-\Delta_B^\sharp-\lambda)_\limminus^\gamma + o(\lambda^{\gamma+ \frac{d-1}2})\quad \mbox{as }\lambda \to \infty\,. 
		\end{equation*}
		\item\label{eq: multicomp subcritical 1} If $\gamma<\gamma_d^\sharp$ and $\Omega_j$ is a connected component of $\widetilde \Omega_j$ for each $j \geq 1$, then
		\begin{equation*}
			\lim_{j\to \infty}|\Omega_j| =0 \quad \mbox{or} \quad \liminf_{j\to \infty} r_{\rm in}(\Omega_j)\sqrt{\lambda_j}<\infty\,.
		\end{equation*}
		In particular, the limit of any subsequence of connected components that converges in $(\mathcal{C}_d, d^H)$ modulo translations is the empty set.
		
		\item\label{eq: multicomp subcritical} If $r_{\gamma+ \frac{1}2,d-1}^\sharp \neq r_{\gamma,d}^\sharp$, then $\widetilde\Omega_j$ consists of $\gtrsim_{\gamma,d} \lambda_j^{d/2}$ disjoint components.
	\end{enumerate}
\end{theorem}

\begin{remark}
	Note that $r_{\gamma+ \frac{1}2,d-1}^\sharp \neq r_{\gamma,d}^\sharp$ for $\gamma \in ((\gamma_{d-1}^\sharp-\frac{1}2)_\limplus, \gamma_d^\sharp]$.
\end{remark}

Note that Theorems \ref{main2} and \ref{main4} in the introduction are consequences of part \ref{eq: multicomp supercritical} of Theorem \ref{thm: shape opt multicomponent}. Indeed, since $\gamma_d^\sharp<1$ by Theorem \ref{smallerthanone}, the assumption $\gamma>\gamma_d^\sharp$ is trivially satisfied for $\gamma\geq 1$. Moreover, the assumption in Theorem \ref{main4} implies $\gamma_d^\sharp =0$, so under this assumption the assumption $\gamma>\gamma_d^\sharp$ is satisfied for any $\gamma>0$.

The remainder of this section is devoted to the proofs of Proposition \ref{prop: energy shape opt multicomponent} and Theorem~\ref{thm: shape opt multicomponent}.

\begin{proof}[Proof of Proposition \ref{prop: energy shape opt multicomponent}]
	By the definition of $r_{\gamma, d}^\sharp$ and the fact that 
    $$
        \Tr(-\Delta_{\Omega_1 \cup \Omega_2}^\sharp-\lambda)_\limminus^\gamma= \Tr(-\Delta_{\Omega_1}^\sharp-\lambda)_\limminus^\gamma+\Tr(-\Delta_{\Omega_2}^\sharp-\lambda)_\limminus^\gamma
    $$ 
    if $\Omega_1, \Omega_2$ are disjoint, we only need to prove a lower bound when $\sharp =\rm D$ and an upper bound when $\sharp =\rm N$. Given $\delta>0$ there exists $(\Omega_*, \lambda_*)\in \mathcal{C}_d\times(0, \infty)$ such that
	\begin{equation*}
		\Biggl|\frac{\Tr(-\Delta_{\Omega_*}^\sharp-\lambda_*)_\limminus^\gamma}{L_{\gamma,d}^{\rm sc}|\Omega_*|\lambda_*^{\gamma+ \frac{d}2}}-r_{\gamma,d}^\sharp \Biggr|\leq \delta\,.
	\end{equation*}
	We construct a trial set $\Omega_\lambda$ for the variational problem defining $\widetilde{M}_\gamma^\sharp(\lambda)$ as follows. Set $r := \bigl(\frac{\lambda_*}{\lambda}\bigr)^{1/2}$, let $M$ be the largest integer so that $M r^d|\Omega_*|\leq 1$, and define $\eta\geq 0$ by $Mr^d|\Omega_*|+\eta|B_1|=1$. Provided $\lambda$ is so large that $M\geq 1$ we define
	\begin{equation*}
		\Omega_\lambda := B_{\eta}(0) \cup\bigl(\cup_{j=1}^M (r\Omega_* + x_j)\bigr)\,,
	\end{equation*}
	where $\{x_j\}_{j=1}^M\subset \R^d$ are chosen so that the $M$ or $M+1$ sets in this union are disjoint. The choices of $M, r, \eta, \{x_j\}_{j=1}^M$ and the convexity of $\Omega_*, B_\eta(0)$ ensures that $\Omega_\lambda$ is a valid trial set in the variational problem defining $\widetilde{M}_\gamma^\sharp(\lambda)$. 
	
	By the behavior of Laplace eigenvalues under dilation of sets, it holds that
	\begin{equation*}
		\Tr(-\Delta_{\Omega_\lambda}^\sharp-\lambda)_\limminus^\gamma = M \Bigl(\frac{\lambda}{\lambda_*}\Bigr)^{\gamma}\Tr(-\Delta_{\Omega_*}^\sharp-\lambda_*)_\limminus^\gamma + \eta^{-2\gamma}\Tr(-\Delta_{B_1(0)}^\sharp- \lambda \eta^2)_\limminus^\gamma\,.
	\end{equation*}
	By definition $\bigl(\frac{\lambda}{\lambda_*}\bigr)^{\frac{d}2}|\Omega_*|^{-1}-1< M \leq \bigl(\frac{\lambda}{\lambda_*}\bigr)^{\frac{d}2}|\Omega_*|^{-1}$ and $0\leq \eta \leq \bigl(\frac{\lambda_*}{\lambda}\bigr)^{\frac{1}2}\bigl(\frac{|\Omega_*|}{|B_1|}\bigr)^{\frac{1}d}$. It follows that
	\begin{equation*}
		\Tr(-\Delta_{\Omega_\lambda}^\sharp-\lambda)_\limminus^\gamma = \lambda^{\gamma+ \frac{d}2}\frac{\Tr(-\Delta_{\Omega_*}^\sharp-\lambda_*)_\limminus^\gamma}{|\Omega_*|\lambda_*^{\gamma+ \frac{d}2}} + O(\lambda^{\gamma}) \quad \mbox{as }\lambda \to \infty\,.
	\end{equation*}
	By the choice of $(\Omega_*, \lambda_*)$ we conclude that
	\begin{equation*}
		\limsup_{\lambda \to \infty} \biggl|\frac{\widetilde{M}_\gamma^\sharp(\lambda)}{L_{\gamma,d}^{\rm sc}\lambda^{\gamma+ \frac{d}2}}-r_{\gamma,d}^\sharp\biggr|\leq \delta\,.
	\end{equation*}
	Since $\delta>0$ was arbitrary, this proves Proposition \ref{prop: energy shape opt multicomponent}.
\end{proof}

The proof of Theorem \ref{thm: shape opt multicomponent} depends on improved semiclassical inequalities above the critical Riesz exponent. The following result is a special case of \cite[Theorems 4.5 and 4.6]{FrankLarson_CPAM26}, where in fact a certain equivalence between $\gamma>\gamma_d^\#$ and the validity of a two-term inequality is proved.

\begin{theorem}\label{thm: improved inequality above critical gamma}
	Let $d\geq 1$.
    \begin{enumerate}
    [label=\textup{(}\hspace{-0.4pt}\alph*\textup{)}]
        \item For any $\gamma>\gamma_d^{\rm D}$ there is a constant $c_{\gamma,d}>0$ such that for all bounded convex open sets $\Omega \subset \R^d$ and $\lambda\geq 0$ we have
		\begin{equation*}
			\Tr(-\Delta_\Omega^{\rm D}-\lambda)_\limminus^\gamma \leq \Bigl(L^{\rm sc}_{\gamma,d} |\Omega| \lambda^{\gamma+\frac d2} -c_{\gamma,d}\Haus^{d-1}(\partial\Omega)\lambda^{\gamma+ \frac{d-1}{2}}\Bigr)_\limplus \,.
		\end{equation*}
        \item For any $\gamma>\gamma_d^{\rm N}$ there is a constant $c_{\gamma,d}>0$ such that for all bounded convex open sets $\Omega \subset \R^d$ and $\lambda\geq 0$ we have
		\begin{equation*}
			\Tr(-\Delta_\Omega^{\rm N}-\lambda)_\limminus^\gamma \geq L^{\rm sc}_{\gamma,d} |\Omega| \lambda^{\gamma+\frac d2} +c_{\gamma,d}\Haus^{d-1}(\partial\Omega)\lambda^{\gamma+ \frac{d-1}{2}} \,.
		\end{equation*}		
    \end{enumerate}
\end{theorem}

In addition, we will use the following two geometric inequalities about bounded convex sets $\Omega\subset\R^d$:
\begin{equation}\label{eq: inradius bound}
	\frac{|\Omega|}{\Haus^{d-1}(\partial\Omega)} \leq r_{\rm in}(\Omega) \leq d\frac{|\Omega|}{\Haus^{d-1}(\partial\Omega)}
	\quad \mbox{and}\quad \diam(\Omega) \leq C_d \frac{|\Omega|}{r_{\rm in}(\Omega)^{d-1}}\,,
\end{equation}
see, for instance,~\cite{LarsonJST} and references therein.

\begin{proof}[Proof of Theorem~\ref{thm: shape opt multicomponent}]
	We split the proof in accordance to the statements in the theorem.
	
	\medskip
	
	{\noindent\it Part 1: Proof of \ref{eq: multicomp supercritical} for $\sharp =\rm D$}: Write $\widetilde{\Omega}_j = \cup_{k\geq 1}\Omega_{j,k}$ with numbering chosen so that $|\Omega_{j,1}|\geq |\Omega_{j,2}| \geq \ldots$ and $\Omega_{j,k}\in \mathcal{C}_d$ and $\Omega_{j,k}\cap \Omega_{j,k'}=\emptyset$ for each $k\neq k'$. Our goal is to prove that $\Omega_{j,1} \to B$ up to translations. (Note that every open subset of $\R^d$ has a countable number of connected components; indeed $\R^d$ is second-countable and so the open cover of $\widetilde\Omega_j$ given by its connected components has a countable sub-cover, since the connected components are disjoint this sub-cover must coincide with the cover itself.) 
	
	Let $\mathcal{K}_{j,0}$ denote the set of indices corresponding to components $\Omega_{j,k}$ of $\widetilde\Omega_j$ satisfying $\Tr(-\Delta_{\Omega_{j,k}}^{\rm D}-\lambda_j)_\limminus^\gamma =0$.
	Since $\gamma >\gamma_d^{\rm D}$, a two-term semiclassical inequality as in Theorem~\ref{thm: improved inequality above critical gamma} holds. Therefore, by applying this inequality to each connected component not in $\mathcal{K}_{j,0}$ we find that
	\begin{align*}
		\Tr(-\Delta_{\widetilde\Omega_j}^{\rm D}-\lambda_j)_\limminus^\gamma 
		&\leq \sum_{k \notin \mathcal{K}_{j,0}}\Bigl(L_{\gamma,d}^{\rm sc}|\Omega_{j,k}|\lambda_j^{\gamma+ \frac{d}2}- c_{\gamma,d} \Haus^{d-1}(\partial\Omega_{j,k})\lambda_j^{\gamma+ \frac{d-1}2}\Bigr)\\
		&=
		L_{\gamma,d}^{\rm sc}\lambda_j^{\gamma+ \frac{d}2}- c_{\gamma,d}\lambda_j^{\gamma+ \frac{d-1}2} \sum_{k \notin \mathcal{K}_{j,0}}\Haus^{d-1}(\partial\Omega_{j,k})- L_{\gamma,d}^{\rm sc}\lambda_j^{\gamma+ \frac{d}2}\sum_{k\in \mathcal{K}_{j,0}}|\Omega_{j,k}|\,.
	\end{align*}
	Using two-term asymptotics for $\Tr(-\Delta_B^{\rm D}-\lambda)_\limminus^\gamma$, the definition of $\widetilde{M}_\gamma^{\rm D}(\lambda_j)$, and the assumed property of $\{\widetilde{\Omega}_j\}_{j\geq 1}$ it follows that
	\begin{align*}
		L_{\gamma,d}^{\rm sc}\lambda_j^{\gamma+ \frac{d}2}& - \frac{L_{\gamma,d-1}^{\rm sc}}{4}\Haus^{d-1}(\partial B)\lambda_j^{\gamma+ \frac{d-1}2} + o(\lambda_j^{\gamma+ \frac{d-1}2})\\
		&\leq L_{\gamma,d}^{\rm sc}\lambda_j^{\gamma+ \frac{d}2}- c_{\gamma,d}\lambda_j^{\gamma+ \frac{d-1}2} \sum_{k \notin \mathcal{K}_{j,0}}\Haus^{d-1}(\partial\Omega_{j,k})- L_{\gamma,d}^{\rm sc}\lambda_j^{\gamma+ \frac{d}2}\sum_{k\in \mathcal{K}_{j,0}}|\Omega_{j,k}|\,.
	\end{align*}
	After rearranging, we deduce that
	\begin{equation}\label{eq: small component bounds 1}
		\sum_{k \in \mathcal{K}_{j,0}}|\Omega_{j,k}| \lesssim_{\gamma,d}\lambda_j^{-\frac{1}2} \quad \mbox{and}\quad  \sum_{k \notin \mathcal{K}_{j,0}}\Haus^{d-1}(\partial\Omega_{j,k}) \lesssim_{\gamma,d} 1\,.
	\end{equation}
	
	For $\Lambda>0$ let $\mathcal{K}_{j,1}(\Lambda)$ denote the set of indices corresponding to components $\Omega_{j,k}$ of $\widetilde\Omega_{j}$ such that $|\Omega_{j,k}|\leq \bigl(\frac{\Lambda}{\lambda_j}\bigr)^{\frac{d}2}$. By the isoperimetric inequality we have
	\begin{equation*}
		\sum_{k \in \mathcal{K}_{j,1}(\Lambda)\setminus \mathcal{K}_{j,0}}\Haus^{d-1}(\partial\Omega_{j,k}) \gtrsim_d  \sum_{k \in \mathcal{K}_{j,1}(\Lambda)\setminus \mathcal{K}_{j,0}}|\Omega_{j,k}|^{\frac{d-1}d}\geq \Bigl(\frac{\Lambda}{\lambda_j}\Bigr)^{-\frac{1}{2}}\sum_{k \in \mathcal{K}_{j,1}(\Lambda)\setminus \mathcal{K}_{j,0}}|\Omega_{j,k}|\,.
	\end{equation*}
	Consequently, by \eqref{eq: small component bounds 1} it follows that
	\begin{equation*}
		\sum_{k \in \mathcal{K}_{j,1}(\Lambda)\setminus \mathcal{K}_{j,0}}|\Omega_{j,k}| \lesssim_{\gamma,d} \Bigl(\frac{\Lambda}{\lambda_j}\Bigr)^{\frac{1}2}\,.
	\end{equation*}
	Asymptotically, almost all of $\widetilde \Omega_j$ (in a measure sense) is made up of connected components whose measure is $\gtrsim \lambda_j^{-\frac{d}2}.$
	
	Fix $\epsilon>0$, by Theorem~\ref{thm: shape optimization convex} we can choose $\Lambda$ so large that $M_{\gamma}^{\rm D}(\lambda) \leq \Tr(-\Delta_B^{\rm D}-\lambda)_\limminus^\gamma +\epsilon\lambda^{\gamma+ \frac{d-1}{2}}$ for all $\lambda \geq \Lambda$. Since $\lambda_j|\Omega_{j,k}|^{\frac{2}d}> \Lambda$ for each $k \notin \mathcal{K}_{j,1}(\Lambda)$, Theorem \ref{thm: shape optimization convex} and the fact that $\gamma >\gamma_d^{\rm D}$ imply that
	\begin{align*}
		\Tr(-\Delta_{\widetilde\Omega_j}^{\rm D}-\lambda_j)_\limminus^\gamma 
		&=\sum_{k \notin \mathcal{K}_{j,1}(\Lambda)\cup \mathcal{K}_{j,0}}\Tr(-\Delta_{\Omega_{j,k}}^{\rm D}-\lambda_j)_\limminus^\gamma+\sum_{k \in \mathcal{K}_{j,1}(\Lambda)\setminus \mathcal{K}_{j,0}}\Tr(-\Delta_{\Omega_{j,k}}^{\rm D}-\lambda_j)_\limminus^\gamma\\
		&\leq\sum_{k \notin \mathcal{K}_{j,1}(\Lambda)\cup \mathcal{K}_{j,0}}|\Omega_{j,k}|^{-\frac{2\gamma}d}\Tr(-\Delta_{|\Omega_{j,k}|^{-\frac{1}d}\Omega_{j,k}}^{\rm D}-\lambda_j|\Omega_{j,k}|^{\frac{2}d})_\limminus^\gamma\\
		&\quad+\sum_{k \in \mathcal{K}_{j,1}(\Lambda)\setminus \mathcal{K}_{j,0}}L_{\gamma,d}^{\rm sc}|\Omega_{j,k}|\lambda_j^{\gamma+ \frac{d}2}\\
		&\leq\sum_{k \notin \mathcal{K}_{j,1}(\Lambda)\cup \mathcal{K}_{j,0}}|\Omega_{j,k}|^{-\frac{2\gamma}d}M_\gamma^{\rm D}(\lambda_j|\Omega_{j,k}|^{\frac{2}d})+\sum_{k \in \mathcal{K}_{j,1}(\Lambda)\setminus \mathcal{K}_{j,0}}L_{\gamma,d}^{\rm sc}|\Omega_{j,k}|\lambda_j^{\gamma+ \frac{d}2}\\
		&\leq
		\sum_{k \notin \mathcal{K}_{j,1}(\Lambda)\cup \mathcal{K}_{j,0}}\Bigl(|\Omega_{j,k}|^{-\frac{2\gamma}d}\Tr(-\Delta_B^{\rm D}-\lambda_j|\Omega_{j,k}|^{\frac{2}d})_\limminus^\gamma+ \epsilon \lambda_j^{\gamma+ \frac{d-1}2}|\Omega_j|^{\frac{d-1}d}\Bigr)\\
		&\quad+\sum_{k \in \mathcal{K}_{j,1}(\Lambda)\setminus \mathcal{K}_{j,0}}L_{\gamma,d}^{\rm sc}|\Omega_{j,k}|\lambda_j^{\gamma+ \frac{d}2}\,.
	\end{align*}
	By the two-term asymptotics for $\Tr(-\Delta_B^{\rm D}-\lambda)_\limminus^\gamma$ (see Theorem \ref{thm: Weyl asymptotics convex}) it follows that
	\begin{align*}
		\Tr(-\Delta_{\widetilde\Omega_j}^{\rm D}-\lambda_j)_\limminus^\gamma
		&\leq
		L_{\gamma,d}^{\rm sc}\lambda_j^{\gamma+ \frac{d}{2}}-\Bigl(\frac{L_{\gamma,d-1}^{\rm sc}}{4}\Haus^{d-1}(\partial B)-\epsilon+o(1)\Bigr)\lambda_j^{\gamma+ \frac{d-1}{2}}\hspace{-11pt} \sum_{k \notin \mathcal{K}_{j,1}(\Lambda)\cup \mathcal{K}_{j,0}}\hspace{-11pt}|\Omega_{j,k}|^{\frac{d-1}d}\,.
	\end{align*}
	
	Using $\Tr(-\Delta_B^{\rm D}-\lambda_j)_\limminus^\gamma \leq \widetilde{M}_\gamma^{\rm D}(\lambda_j)\leq \Tr(-\Delta_{\widetilde\Omega_j}^{\rm D}-\lambda_j)_\limminus^\gamma+ o(\lambda_j^{\gamma+ \frac{d-1}2})$ and again the two-term Weyl asymptotics for $\Tr(-\Delta_B^{\rm D}-\lambda_j)_\limminus^\gamma$ we conclude that 
	\begin{equation*}
		\limsup_{j\to \infty} \sum_{k \notin \mathcal{K}_{j,1}(\Lambda)\cup \mathcal{K}_{j,0}}|\Omega_{j,k}|^{\frac{d-1}d} \leq 1- \frac{4}{L_{\gamma,d-1}^{\rm sc}\Haus^{d-1}(\partial B)}\epsilon\,.
	\end{equation*}
	Since
	\begin{equation*}
		\sum_{k \notin \mathcal{K}_{j,1}(\Lambda)\cup \mathcal{K}_{j,0}}|\Omega_{j,k}|^{\frac{d-1}d} \geq \Bigl(\min_{k \notin \mathcal{K}_{j,1}(\Lambda)\cup \mathcal{K}_{j,0}}|\Omega_{j,k}|\Bigr)^{-\frac{1}d} \sum_{k \notin \mathcal{K}_{j,1}(\Lambda)\cup \mathcal{K}_{j,0}}|\Omega_{j,k}|
	\end{equation*}
	and 
	$$
	\liminf_{j\to \infty}\sum_{k \notin \mathcal{K}_{j,1}(\Lambda)\cup \mathcal{K}_{j,0}}|\Omega_{j,k}| = 1 \,,
	$$ 
	it follows that
	\begin{equation*}
		\liminf_{j\to \infty} \min_{k \notin \mathcal{K}_{j,1}(\Lambda)\cup \mathcal{K}_{j,0}} |\Omega_{j,k}| \geq 1- c_{\gamma,d}\epsilon\,.
	\end{equation*}
	Since $\epsilon>0$ was arbitrary, it follows that in the limit $\widetilde\Omega_j$ has a connected component whose measure converges to $1$, that is, $|\Omega_{j,1}|= 1+o(1)$ as $j \to \infty$. As we have shown that $\Haus^{d-1}(\partial\Omega_{j,1})\leq C$, it follows from \eqref{eq: inradius bound} that $\inf_{j\geq 1}r_{\rm in}(\Omega_{j,1})>0$ and $\sup_{j\geq 1}\diam(\Omega_{j,1})<\infty$. In particular, by Theorem \ref{thm: Weyl asymptotics convex} it holds that
	\begin{equation*}
		\Tr(-\Delta_{\Omega_{j,1}}^{\rm D}- \lambda_j)_\limminus^{\rm D} = L_{\gamma,d}^{\rm sc}|\Omega_{j,1}|\lambda_{j}^{\gamma+ \frac{d}2} - \frac{L_{\gamma,d-1}^{\rm sc}}{4}\Haus^{d-1}(\partial\Omega_{j,1})\lambda_j^{\gamma+ \frac{d-1}2}+ o(\lambda_j^{\gamma+ \frac{d-1}2})\,.
	\end{equation*}
	Consequently, 
	\begin{align*}
		\Tr(-\Delta_{\widetilde\Omega_j}^{\rm D}-\lambda_j)_\limminus^\gamma
		&= \sum_{k\geq 1}\Tr(-\Delta_{\Omega_{j,k}}^{\rm D}-\lambda_j)_\limminus^\gamma\\
		&\leq \Tr(-\Delta_{\Omega_{j,1}}^{\rm D}-\lambda_j)_\limminus^\gamma + L_{\gamma,d}^{\rm sc}(1-|\Omega_{j,1}|)\lambda_j^{\gamma+ \frac{d}2}\\
		&\leq  L_{\gamma,d}^{\rm sc}\lambda_{j}^{\gamma+ \frac{d}2} - \frac{L_{\gamma,d-1}^{\rm sc}}{4}\Haus^{d-1}(\partial\Omega_{j,1})\lambda_j^{\gamma+ \frac{d-1}2}+ o(\lambda_j^{\gamma+ \frac{d-1}2})\,.
	\end{align*}
	By again comparing with $B$ we find
	\begin{equation*}
		\limsup_{j\to \infty}\Haus^{d-1}(\partial\Omega_{j,1})\leq \Haus^{d-1}(\partial B)\,. 
	\end{equation*}
	Since $|\Omega_{j,1}|\to 1$, Blaschke's selection theorem together with the isoperimetric inequality implies that up to translation $\Omega_{j,1}\to B$ with respect to the Hausdorff distance.
	
	Since perimeter is continuous with respect to convergence of convex sets in the Hausdorff distance, it follows that $\Haus^{d-1}(\partial\Omega_{j,1})= \Haus^{d-1}(\partial B)+o(1)$. By one final time comparing with the ball and using $\gamma>\gamma_d^{\rm D}$ to bound the contributions from $\Omega_{j,k}$ with $k \geq 1$, we find that
	\begin{align*}
		L_{\gamma,d}^{\rm sc}\lambda_j^{\gamma+ \frac{d}2} &- \frac{L_{\gamma,d-1}^{\rm sc}}{4}\Haus^{d-1}(\partial B)\lambda_j^{\gamma+ \frac{d-1}2} + o(\lambda_j^{\gamma+ \frac{d-1}2})\\
		&\leq \Tr(-\Delta_{\widetilde\Omega_j}^{\rm D}-\lambda_j)_\limminus^\gamma \\
		&\leq L_{\gamma,d}^{\rm sc}\lambda_j^{\gamma+ \frac{d}2}- \frac{L_{\gamma,d-1}^{\rm sc}}{4}\Haus^{d-1}(\partial B)\lambda_j^{\gamma+ \frac{d-1}2}+ o(\lambda_j^{\gamma+ \frac{d-1}{2}})\,.
	\end{align*}
	Consequently,
	\begin{equation*}
		\Tr(-\Delta_{\widetilde\Omega_j}^{\rm D}-\lambda_j)_\limminus^\gamma = \Tr(-\Delta_{B}^{\rm D}-\lambda_j)_\limminus^\gamma + o(\lambda_j^{\gamma+ \frac{d-1}2})\,.
	\end{equation*}
	This completes the proof of \ref{eq: multicomp supercritical} for $\sharp =\rm D$.
	
	\medskip
	
	{\noindent\it Part 2: Proof of \ref{eq: multicomp supercritical} for $\sharp =\rm N$}: 
	The idea of the proof is the same as in the Dirichlet case. The main differences are in the proof of that there is a single connected component that asymptotically carries all of the measure.
	
	Since $0$ is an eigenvalue of $-\Delta_{\widetilde\Omega_j}^{\rm N}$ with multiplicity given by the number of connected components of $\widetilde\Omega_j$, we deduce from the assumed property of $\widetilde \Omega_j$ that the number of connected components is controlled. Indeed, by Proposition \ref{prop: energy shape opt multicomponent} we deduce that
	\begin{equation}\label{eq: number of components Neu}
		\#\{\mbox{connected components of }\widetilde\Omega_j\} \lambda_j^\gamma\leq \Tr(-\Delta_{\widetilde \Omega_j}^{\rm N}-\lambda_j)_\limminus^\gamma = L_{\gamma,d}^{\rm sc}\lambda_j^{\gamma+ \frac{d}2} + o(\lambda_j^{\gamma+ \frac{d}2})\,.
	\end{equation}
	
	Write $\widetilde{\Omega}_j = \cup_{k\geq 1}\Omega_{j,k}$ with numbering chosen so that $|\Omega_{j,1}|\geq |\Omega_{j,2}| \geq \ldots$ and $\Omega_{j,k}\in \mathcal{C}_d$ and $\Omega_{j,k}\cap \Omega_{j,k'}=\emptyset$ for each $k\neq k'$. Our goal is to prove that $\Omega_{j,1} \to B$ up to translations.
	
	Since $\gamma >\gamma_d^{\rm N}$ a two-term semiclassical inequality as in Theorem~\ref{thm: improved inequality above critical gamma} holds. Therefore, by applying this inequality to each connected component we find that
	\begin{align*}
		\Tr(-\Delta_{\Omega_j}^{\rm N}-\lambda_j)_\limminus^\gamma 
		&\geq \sum_{k\geq 1}\Bigl(L_{\gamma,d}^{\rm sc}|\Omega_{j,k}|\lambda_j^{\gamma+ \frac{d}2}+ c_{\gamma,d} \Haus^{d-1}(\partial\Omega_{j,k})\lambda_j^{\gamma+ \frac{d-1}2}\Bigr)\\
		&=
		L_{\gamma,d}^{\rm sc}\lambda_j^{\gamma+ \frac{d}2}+ c_{\gamma,d}\lambda_j^{\gamma+ \frac{d-1}2}\Haus^{d-1}(\partial\widetilde\Omega_{j})\,.
	\end{align*}
	Using two-term asymptotics for $\Tr(-\Delta_B^{\rm N}-\lambda)_\limminus^\gamma$, the definition of $\widetilde{M}_\gamma^{\rm N}(\lambda_j)$, and the assumed property of $\{\widetilde{\Omega}_j\}_{j\geq 1}$ it follows that
	\begin{align*}
		L_{\gamma,d}^{\rm sc}\lambda_j^{\gamma+ \frac{d}2}& + \frac{L_{\gamma,d-1}^{\rm sc}}{4}\Haus^{d-1}(\partial B)\lambda_j^{\gamma+ \frac{d-1}2} + o(\lambda_j^{\gamma+ \frac{d-1}2})\\
		&\geq L_{\gamma,d}^{\rm sc}\lambda_j^{\gamma+ \frac{d}2}+ c_{\gamma,d}\lambda_j^{\gamma+ \frac{d-1}2} \Haus^{d-1}(\partial\widetilde\Omega_{j})\,.
	\end{align*}
	After rearranging, we deduce that
	\begin{equation}\label{eq: small component bounds 1 Neu}
		\Haus^{d-1}(\partial\widetilde\Omega_{j}) \lesssim_{\gamma,d} 1\,.
	\end{equation}
	
	For $\Lambda>0$ let $\mathcal{K}_{j}(\Lambda)$ denote the set of indices corresponding to components $\Omega_{j,k}$ of $\widetilde\Omega_{j}$ such that $|\Omega_{j,k}|\leq \bigl(\frac{\Lambda}{\lambda_j}\bigr)^{\frac{d}2}$. By the isoperimetric inequality we have
	\begin{equation*}
		\sum_{k \in \mathcal{K}_{j}(\Lambda)}\Haus^{d-1}(\partial\Omega_{j,k}) \gtrsim_d  \sum_{k \in \mathcal{K}_{j}(\Lambda)}|\Omega_{j,k}|^{\frac{d-1}d}\geq \Bigl(\frac{\Lambda}{\lambda_j}\Bigr)^{-\frac{1}{2}}\sum_{k \in \mathcal{K}_{j}(\Lambda)}|\Omega_{j,k}|\,.
	\end{equation*}
	Consequently, by \eqref{eq: small component bounds 1 Neu} it follows that
	\begin{equation*}
		\sum_{k \in \mathcal{K}_{j}(\Lambda)}|\Omega_{j,k}| \lesssim_{\gamma,d} \Bigl(\frac{\Lambda}{\lambda_j}\Bigr)^{\frac{1}2}\,.
	\end{equation*}
	That is, we have shown that asymptotically all the measure of $\widetilde\Omega_j$ is located in connected components whose measure is $\gtrsim \lambda_j^{-\frac{d}2}$. 
	
	The proof of \ref{eq: multicomp supercritical} with $\sharp =\rm N$ can now be completed by following the same strategy as in the Dirichlet case. We omit the details. 
	
	\medskip
	
	{\noindent\it Part 3: Proof of \ref{eq: multicomp subcritical 1}}: Let $\widetilde\Omega_j$ be as in the statement of the theorem and write as before $\widetilde\Omega_j = \cup_{k\geq 1}\Omega_{j,k}$. Let $\{k_j\}_{j\geq 1}$ be the sequence of indices so $\Omega_j=\Omega_{j,k_j}$ is the sequence of connected components in the statement. We shall argue that $\liminf_{j \to \infty} r_{\rm in}(\Omega_{j,k_j})\sqrt{\lambda_j} =\infty$ implies that $\lim_{j\to \infty}|\Omega_{j,k_j}|=0$. 
	
	Assume that $\liminf_{j\to \infty}r_{\rm in}(\Omega_{j,k_k})\sqrt{\lambda_j}=\infty$. Since $\Haus^{d-1}(\partial\Omega_j) \leq \frac{d|\Omega_j|}{r_{\rm in}(\Omega_j)}$ by \eqref{eq: inradius bound}, Theorem \ref{thm: Weyl asymptotics convex} (if $\gamma>0$) or \cite[Theorem 5.2]{FrankLarson_CPAM26} (if $\gamma=0$) implies that
	\begin{equation*}
		\Tr(-\Delta_{\Omega_{j,k_j}}^\sharp-\lambda_j)_\limminus^\gamma = L_{\gamma,d}^{\rm sc}|\Omega_{j,k_j}|\lambda_j^{\gamma+\frac{d}2}(1+o(1))
	\end{equation*}
	as $j \to \infty$. Therefore, by the definition of $r_{\gamma,d}^{\rm D}$ we have
	\begin{align*}
		\Tr(-\Delta_{\widetilde\Omega_j}^{\rm D}-\lambda_j)_\limminus^\gamma
		&=
		\Tr(-\Delta_{\Omega_{j,k_j}}^{\rm D}-\lambda_j)_\limminus^\gamma+\sum_{k\neq k_j} \Tr(-\Delta_{\Omega_{j,k}}^{\rm D}-\lambda_j)_\limminus^\gamma\\
		&\leq
		L_{\gamma,d}^{\rm sc}|\Omega_{j,k_j}|\lambda_j^{\gamma+ \frac{d}{2}}(1+o_{j\to \infty}(1))
		+\sum_{k\in \mathcal{K}_j(\Lambda)} r_{\gamma,d}^{\rm D}L_{\gamma,d}^{\rm sc}|\Omega_{j,k}|\lambda_j^{\gamma+ \frac{d}2}\\
		&=
		L_{\gamma,d}^{\rm sc}|\Omega_{j,k_j}|\lambda_j^{\gamma+ \frac{d}{2}}(1+o_{j\to \infty}(1))
		+r_{\gamma,d}^{\rm D}L_{\gamma,d}^{\rm sc}(1-|\Omega_{j,k_j}|)\lambda_j^{\gamma+ \frac{d}2}\,.
	\end{align*}
	As $r_{\gamma,d}^{\rm D}>1$ as $\gamma<\gamma_d^{\rm D}$ this contradicts Proposition \ref{prop: energy shape opt multicomponent} unless $\lim_{j\to \infty}|\Omega_{j,k_j}|=0$. The corresponding conclusion in the Neumann case can be proved in the same manner.
	
	Finally, if after translations a subsequence of the connected $\{\Omega_j\}_{j\geq 1}$ had a non-empty limit $\Omega_\infty \in \mathcal{C}_d$ the continuity of the volume and inradius in $(\mathcal{C}_d, d^H)$ away from the emptyset the first statement in \ref{eq: multicomp subcritical 1} in would imply that either $|\Omega_\infty|=0$ or $r_{\rm in}(\Omega_\infty)=0$ neither of which can be true for a set $\Omega_\infty \in \mathcal{C}_d$. This concludes the proof of \ref{eq: multicomp subcritical 1}.
	
	\medskip
	{\noindent\it Part 4: Proof of \ref{eq: multicomp subcritical}}: Let $\widetilde\Omega_j$ be as in the statement of the theorem and write as before $\widetilde\Omega_j = \cup_{k\geq 1}\Omega_{j,k}$. For $\Lambda>0$, denote by $\mathcal{K}_{j}(\Lambda)$ the collection of indices $k$ so that $|\Omega_{j,k}|\leq \bigl(\frac{\Lambda}{\lambda_j}\bigr)^{\frac{d}2}$. For a given $\epsilon>0$ choose $\Lambda$ so large that Proposition~\ref{prop:shapeoptasymp r} implies that
	\begin{equation*}
		\biggl|\frac{M_\gamma^\sharp(\lambda)}{r_{\gamma+ \frac{1}2,d-1}^\sharp L_{\gamma,d}^{\rm sc}\lambda^{\gamma+ \frac{d}{2}}} - 1 \biggr| \leq \epsilon \quad \mbox{for all } \lambda \geq \Lambda\,.
	\end{equation*}
	Let us consider the Dirichlet case, the Neumann case can be treated identically. By scaling of Laplacian eigenvalues, the bound on $M^{\rm D}_\gamma$ above, and the definition of $r_{\gamma,d}^{\rm D}$ yield
	\begin{align*}
		\Tr(-\Delta_{\widetilde\Omega_j}^{\rm D}-\lambda_j)_\limminus^\gamma
		&=
		\sum_{k\notin \mathcal{K}_j(\Lambda)} \Tr(-\Delta_{\Omega_{j,k}}^{\rm D}-\lambda_j)_\limminus^\gamma+\sum_{k\in \mathcal{K}_j(\Lambda)} \Tr(-\Delta_{\Omega_{j,k}}^{\rm D}-\lambda_j)_\limminus^\gamma\\
		&\leq
		\sum_{k\notin \mathcal{K}_j(\Lambda)} r_{\gamma+\frac{1}2,d-1}^{\rm D}L_{\gamma,d}^{\rm sc}(1+\epsilon)|\Omega_{j,k}|\lambda_j^{\gamma+ \frac{d}{2}}
		+\sum_{k\in \mathcal{K}_j(\Lambda)} r_{\gamma,d}^{\rm D}L_{\gamma,d}^{\rm sc}|\Omega_{j,k}|\lambda_j^{\gamma+ \frac{d}2}\,.
	\end{align*}
	If we choose $\epsilon$ so small that $r_{\gamma+\frac12,d-1}^{\rm D}(1+\epsilon) < r_{\gamma,d}^{\rm D}$, this contradicts Proposition \ref{prop: energy shape opt multicomponent} unless
	\begin{equation*}
		\sum_{k\notin \mathcal{K}_j(\Lambda)}|\Omega_{j,k}| = o(1) \quad \mbox{as }j \to \infty\,.
	\end{equation*}
	In particular, asymptotically all of the mass of $\widetilde\Omega_j$ is concentrated in components whose individual measure is $\lesssim_{\gamma,d}\lambda_j^{-\frac{d}2}$ so there must be $\gtrsim_{\gamma,d}\lambda_j^{\frac{d}2}$ such components. For $\sharp = {\rm N}$ we note that this is estimate is order-sharp by the upper bound in \eqref{eq: number of components Neu}.
\end{proof}


\bibliographystyle{amsalpha}

\end{document}